\documentclass[12pt]{article}
\usepackage{amsmath,amssymb,amsthm}
\allowdisplaybreaks
\def\cc{\mathbb{C}}
\def\M{\mathcal M}
\def\N{\mathcal N}

\def\cc{\mathbb{C}}

\def\B{\mathcal B}

\def\B{\mathcal B}

\def\H{\mathcal H}

\def\L{\mathcal L}
\def\M{\mathcal M}
\def\N{\mathcal N}

\def\R{\mathcal R}

\newtheorem{set}{set}[section]
\newtheorem{Corollary}[set]{Corollary}

\newtheorem{Lemma}[set]{Lemma}
\newtheorem{Proposition}[set]{Proposition}
\newtheorem{Remark}[set]{Remark}
\newtheorem{Theorem}[set]{Theorem}

\numberwithin{equation}{section}
\begin{document}
\pagestyle{myheadings}
\date{}
\title{A note on the invariant subspace problem relative to a type ${\rm II}_1$ factor}
\author{Junsheng Fang \and Don Hadwin}
\maketitle
\begin{abstract}
Let $\M$ be a type ${\rm II}_1$ factor with a faithful normal
tracial state $\tau$ and let $\M^\omega$ be the ultrapower algebra
of $\M$.   In this paper, we prove that for every operator $T\in
\M^\omega$, there is a  family of projections $\{P_t\}_{0\leq
t\leq 1}$ in $\M^\omega$ such that $TP_t=P_tTP_t$, $P_s\leq P_t$
if $s\leq t$, and $\tau_\omega(P_t)=t$. Let $\mathfrak{M}=\{Z \in
\M:\, \text{there is a  family of projections}\, \{P_t\}_{0\leq
t\leq 1}\newline \text{in}\, \M\, \text{such that}\, ZP_t=P_tZP_t,
P_s\leq P_t\, \text{if}\, s\leq t,\, \text{and}\, \tau(P_t)=t\}$.
As an application we show that for every operator $T\in \M$ and
$\epsilon>0$, there is an operator $S\in \mathfrak{M}$ such that
$\|S\|\leq \|T\|$ and $\|S-T\|_2<\epsilon$.
 We also show that $\prod_n^\omega M_n(\cc)$
is not $\ast$-isomorphic to the ultrapower algebra of the
hyperfinite type ${\rm II}_1$ factor.
\end{abstract}

\section{Introduction}
Let $\M$ be a type ${\rm II}_1$ factor acting on a Hilbert space
$\H$. The invariant subspace problem relative to a factor von
Neumann algebra $\M$ asks for every operator $T\in \M$, does there
exists a projection $P\in \M$, $0<P<I$, such that $TP=PTP$. The
hyperinvariant subspace problem relative to $\M$ asks for every
operator $T\in \M\setminus\cc I$, does there exists a projection
$P$, $0<P<I$, such that $SP=PSP$ for every operator $S$ in
$\B(\H)$ with $ST=TS$. It is easy to see that if a projection $P$
is hyperinvariant for $T$, then $P$ is in the von Neumann algebra
generated by $T$ and therefore in $\M$. A huge advance on the
(hyper)invariant subspace problem relative to a factor of type
${\rm II}_1$ has been made during past ten years (see for
example~\cite{D-H1,D-H2,H-S1,S-S}).

In 1983, Brown~\cite{Br} introduced a spectral distribution
measure for non-normal elements in a finite von Neumann algebra
with respect to a fixed normal faithful tracial state, which is
called the Brown measure of the operator. Recently, Haagerup and
Schultz~\cite{H-S1} proved a remarkable result which states that
if the support of Brown measure of an operator in a type ${\rm
II}_1$ factor contains more than two points, then the operator has
a non-trivial hyperinvariant subspace affiliated with the type
${\rm II}_1$ factor.   However, the  invariant subspace problem
relative to a type ${\rm II}_1$ factor still remains open for
operators with single point Brown measure support (for this case,
we refer to Dykema and Haagerup's paper~\cite{D-H1}).

Suppose that each $\M_n$ is a finite von Neumann algebra with a
faithful normal tracial state $\tau_n$. Let $\prod_{n\in\N}\M_n$
be the $l^\infty$-product of the $\M_n$'s. Then $\prod_n\M_n$ is a
von Neumann algebra (with pointwise multiplication). Let $\omega$
be a free ultrafilter on $\N$ ($\omega$ may be viewed as an
element in $\beta\N\setminus\N$, where $\beta\N$ is the
Stone-C\'ech compactification of $\N$).  If $\{X_n\}$ and
$\{Y_n\}$ are two elements in $\prod_n\M_n$, then we define
$\{X_n\}\sim \{Y_n\}$ when $\lim_{n\to\omega} \|X_n-Y_n\|_2=0$.
Recall that for an operator $T_n\in \M_n$,
$\|T_n\|_2=\tau_n(T_n^*T_n)^{1/2}$. Then the {\it ultraproduct\/},
denoted by $\prod^\omega\M_n$, of $\M_n$ (with respect to the free
ultrafilter $\omega$) is the quotient von Neumann algebra of
$\prod_n\M_n$ modulo the equivalence relation $\sim$ and the limit
of $\tau_n$ at $\omega$ gives rise to a tracial state on
$\prod^\omega\M_n$. We shall use $\tau_\omega$  to denote the
tracial state on $\prod^\omega\M_n$. When $\M_n=\M$ for all $n$,
then $\prod^\omega\M_n$ is called the {\it ultrapower\/} of $\M$,
denoted by $\M^\omega$. The initial algebra $\M$ is embedded into
$\M^\omega$ as constant sequences given by elements in $\M$.
 Ultrapowers for finite von Neumann algebras were
first introduced and studied by  McDuff~\cite{Mc}. Sakai~\cite{Sa}
showed that an ultrapower of a finite von Neumann algebra with
respect to a faithful normal trace is again a finite von Neumann
algebra, and the ultrapower algebra $\M^\omega$ of a type ${\rm
II}_1$ factor is also a type ${\rm II}_1$ factor. Ultrapowers of
type ${\rm II}_1$ factors play an important role in
the study of type ${\rm II}_1$ factors. \\

This paper is organized as  follows. In section 2 of this paper,
we prove that every operator in an ultrapower algebra of a type
${\rm II}_1$ factor $\M$ has a nontrivial invariant space
affiliated with the ultrapower algebra. Precisely, we prove that
for every operator $T\in \M^\omega$, there is a  family of
projections $\{P_t\}_{0\leq t\leq 1}$ in $\M^\omega$ such that
$TP_t=P_tTP_t$, $P_s\leq P_t$ if $s\leq t$, and
$\tau_\omega(P_t)=t$. This result is more or less trivial if $\M$
 has  property $\Gamma$. Recall that $\M$ is said to have property
$\Gamma$ if for any finite elements $T_1,\cdots, T_n$ in $\M$ and
$\epsilon>0$, there is a unitary operator $U$ in $\M$ such that
$\tau(U)=0$ and $\|T_iU-UT_i\|_2<\epsilon$ for $1\leq i\leq n$. If
$\M$ is a separable (with separable predual) type ${\rm II}_1$
factor, then $\M$ has property $\Gamma$ if and only if $\M'\cap
\M^\omega$ is non-trivial. Dixmier [Di] proved that if $\M'\cap
\M^\omega$ is non-trivial, then it is non-atomic. This implies
that if $\M$ has property $\Gamma$, then for every operator $T\in
\M^\omega$, there is a  family of projections $\{P_t\}_{0\leq
t\leq 1}$ in $\M^\omega$ such that $TP_t=P_tT$, $P_s\leq P_t$ if
$s\leq t$, and $\tau_\omega(P_t)=t$. To prove the result for
non-$\Gamma$ factors, we need combine techniques developed by
Haagerup and Schultz~\cite{H-S1}
and a result of  Popa~\cite{Po}.\\

As an application, in section 3 we show that for every operator
$T$ in the unit ball of $\M$ and $\epsilon>0$, there is an
operator $S\in \mathfrak{M}$ such that $\|S\|\leq 1$ and
$\|S-T\|_2<\epsilon$, where $\mathfrak{M}=\{Z \in \M:\,
\text{there is a  family of projections}\, \{P_t\}_{0\leq t\leq
1}\, \text{in}\, \M\, \text{such that}\, ZP_t=P_tZP_t, P_s\leq
P_t\, \text{if}\, s\leq t,\, \text{and}\, \tau(P_t)=t\}$. In
particular, this implies that $\mathfrak{M}$ is dense in $\M$ in
the strong operator topology.\\

In section 4, we give a very simple proof of $\prod_n^\omega M_n(\cc)$ is not
$\ast$-isomorphic to the ultrapower algebra of the hyperfinite
type ${\rm II}_1$ factor (this result might be known to specialists,
 however we can not find it in the existed literature). This result relies on a result
of Herrero and Szarek~\cite{H-S} (also see~\cite{Von}).\\

Thanks to the existence of a faithful normal tracial state on a
type ${\rm II}_1$ factor, in section 5 we show that if two
operators $S$ and $T$ are quasi-similar in a type ${\rm II}_1$
factor $\M$, then $Lat S\cap \M$ is not trivial if and only if
$Lat T\cap \M$ is not trivial. As a corollary, we show that for
two operator $S,T$ in $\M$, $Lat (ST)\cap \M$ is not trivial if
and only if $Lat (TS)\cap \M$ is not trivial. On the other hand,
if the same result also holds for arbitrary two operators in
$\B(\H)$, then the answer to the classical invariant subspace
problem is affirmative (see Remark~\ref{R:last remark}).\\

\noindent \emph{Acknowledgment:} The authors thank David Sherman for his comments on Lemma 4.2 and Theorem 4.3 and for poiting out to us von Neumann's paper~\cite{Von}.

\section{Invariant subspaces for  operators in the ultrapower algebras}

The main result of this section is the following result.

\begin{Theorem}\label{T:main result} Let $\M$ be a type ${\rm II}_1$ factor and let $\M^\omega$ be the
ultrapower algebra of $\M$. For every operator $T\in \M^\omega$,
there is a  family of projections $\{P_t\}_{0\leq t\leq 1}$ in
$\M^\omega$ such that $TP_t=P_tTP_t$, $P_s\leq P_t$ if $s\leq t$,
and $\tau_\omega(P_t)=t$.
\end{Theorem}

\begin{Corollary}\label{C:main result} Let $\M$ be a type ${\rm II}_1$ factor with
a faithful normal tracial state $\tau$. For every operator $T\in
\M$ and $0\leq t\leq 1$,  there is a sequence of projections
$P_n\in \M$ such that
$\lim_{n\rightarrow\infty}\|TP_n-P_nTP_n\|_2=0$ and $\tau(P_n)=t$.
\end{Corollary}

To prove Theorem~\ref{T:main result}, we need the following
lemmas.\\

Let $\M$ be a type ${\rm II}_1$ factor and let $T\in\M$. We regard
$\mathcal{M}$ as a subfactor of $\mathcal{M}_1=\mathcal{M}\ast
\L(\mathbb{F}_4)$. The faithful normal tracial state on $\M_1$
will also be denoted by $\tau$. We choose a circular system
$\{x,y\}$ (in the sense of~\cite{VDN}) that generates
$\L(\mathbb{F}_4)$ and which therefore is free from $\M$. By
Theorem 5.2 of~\cite{H-S2}, the unbounded operator $z=xy^{-1}$ is
in $L^p(\M_1,\tau)$ for $0<p<1$. Let $T_n=T+\frac{1}{n}z$. Then
$T_n\in L^p(\M_1,\tau)$ for $0<p<1$. We will need the following
lemma, which follows from Proposition 4.5, Corollary 4.6, Theorem
5.1 and Theorem 6.9 of~\cite{H-S1}.
\begin{Lemma}\label{L:H-S}With the above assumption, we have
\begin{enumerate}
\item $\lim_{n\rightarrow \infty}\|T-T_n\|_p^p=0$;

\item for every $n$, there is a projection $P_n\in \M_1$ such that
$T_nP_n=P_nT_nP_n$ and $\tau(P_n)=\frac{1}{2}$.
\end{enumerate}
\end{Lemma}

The next lemma follows from the main theorem of~\cite{Po}.
\begin{Lemma}\label{L:popa} Let $\M$ be a separable type ${\rm II}_1$ factor. Then there
is a unitary operator $u\in \M^\omega$ such that
\[\{\M, u\M u^*\}''\cong\M * (u\M u^*).
\]
\end{Lemma}

\begin{Lemma}\label{L:one half} Let $\M$ be a separable type ${\rm II}_1$ factor and let $T\in\M$.
 Then for every $\epsilon>0$, there is a
projection $P\in \M$, $\tau(P)=1/2$, such that
$\|TP-PTP\|_2<\epsilon$.
\end{Lemma}
\begin{proof}
 Note that $\M$ is a von Neumann subalgebra of $\M^\omega$ if
we identify $T\in \M$ with the constant sequence $(T)\in
\M^\omega$. To prove the lemma, it is sufficient  to show that
there is a projection $P\in \M^\omega$, $\tau(P)=1/2$, such that
$\|TP-PTP\|_2<\epsilon$. By Lemma~\ref{L:popa}, there is a unitary
operator $u\in \M^\omega$ such that $\{\M, u\M u^*\}''\cong\M *
(u\M u^*)$. So it is sufficient to show that  there is a
projection $P\in \{\M,u\M u^*\}''$, $\tau(P)=1/2$, such that
$\|TP-PTP\|_2<\epsilon$. Note that $T\in \M$ and therefore $T$ is
free with $u\M u^*$ in $\{\M,u\M u^*\}''$.
 Repeat the above arguments twice if
necessary, we may assume that
 $\M\supseteq \L(\mathbb{F}_4)$ and $T$ is free with
 $\L(\mathbb{F}_4)$.\\

  We choose a circular system $\{x,y\}$ in
 $\L(\mathbb{F}_4)$. Let $z=xy^{-1}$ and $T_n=T+\frac{1}{n}z$.
By Lemma~\ref{L:H-S}, for every $n\geq 1$, there is a projection
$P_n\in \M$ with $\tau(P_n)=1/2$ and $T_nP_n=P_nT_nP_n$. By
Lemma~\ref{L:H-S}, $\lim_{n\rightarrow\infty}\|T_n-T\|_p^p=0$ for
$0<p<1$. Note that
\begin{eqnarray*}
  \|P_n T P_n-T P_n\|_2^2 &=& \tau(|P_nTP_n-TP_n|^2) \\
   &=& \tau(|P_nTP_n-TP_n|^{p/2}|P_nTP_n-TP_n|^{2-p/2}) \\
   &\leq& \tau(|P_nTP_n-TP_n|^p)^{1/2}\tau(|P_nTP_n-TP_n|^{4-p})^{1/2} \\
   &=& \|P_nTP_n-TP_n\|_p^{p/2} \|P_nTP_n-TP_n\|_{4-p}^{(4-p)/2} \\
   &\leq& \left(\|P_nTP_n-TP_n\|_p^p\right)^{1/2}
   \|2T\|_{4-p}^{(4-p)/2},
\end{eqnarray*}
and
\begin{eqnarray*}
  \|P_nTP_n-TP_n\|_p^p &\leq& \|P_n(T-T_n)P_n-(T-T_n)P_n\|_p^p  \\
   &\leq & \|P_n(T-T_n)P_n\|_p^p+(T-T_n)P_n\|_p^p \\
   &\leq& 2\|T-T_n\|_p^p\rightarrow 0.
\end{eqnarray*}
Therefore, $\lim_{n\rightarrow\infty}\|P_n T P_n-T P_n\|_2^2=0$.

\end{proof}

\begin{Lemma}\label{L:general r} Let $\M$ be a separable type
 ${\rm II}_1$ factor, $T\in\M$ and $\epsilon>0$. For every
 positive integer $n$, there are projections $\{P_j\}_{j=0}^{2^n}$
 in $\M$ such that $0=P_0<P_1< P_2< \cdots<P_{2^n-1}< P_{2^n}=I$,
 $\tau(P_j)=j/2^n$, and $\|TP_j-P_jTP_j\|_2\leq \epsilon$ for all $0\leq j\leq 2^n$.
\end{Lemma}
\begin{proof}If $n=1$, then the lemma follows from Lemma~\ref{L:one half}.
 Suppose $n=2$. By Lemma~\ref{L:one half}, there are projections
$P,Q$ in $\M$ such that $\tau(P)=\tau(Q)=1/2$, $P+Q=1$ and
$\|TP-PTP\|_2<\epsilon/2$. Let $a=PTP$, $b=PTQ$, $c=QTP$, and
$d=QTQ$. We can write
\[T=\left(\begin{array}{cc}
a&b\\
c&d \end{array}\right)
\] with respect to the decomposition $I=P+Q$.  Then
$\|c\|_2<\epsilon/2$. Note that both $P\M P$ and $Q\M Q$ are type
${\rm II}_1$ factors. We  apply Lemma~\ref{L:one half} to $a\in
P\M P$ and $b\in Q\M Q$, respectively. There are projections
$P_1\leq P$, $Q_1\leq Q$ such that $\tau(P_1)=\tau(Q_1)=1/4$,
$\|aP_1-P_1aP_1\|_2<\epsilon/2$ and
$\|bQ_1-Q_1bQ_1\|_2<\epsilon/2$. Let $P_0=0$, $P_2=P$,
$P_3=P+Q_1$, and $P_4=I$. Then $0=P_0<P_1<P_2<P_3<P_4=I$ and
$\tau(P_j)=j/4$ for $0\leq j\leq 4$. Simple computations show that
 $\|TP_j-P_jTP_j\|_2\leq \epsilon$ for all $0\leq
j\leq 4$. The general case can be proved by using the induction on
$n$ with similar arguments as the above.
\end{proof}

Combining Lemma~\ref{L:general r} and the noncommutative
H$\ddot{\text{o}}$lder's inequality, we have the following:
\begin{Corollary}\label{C:arbitrary t} Let $\M$ be a separable type ${\rm II}_1$ factor and let $T\in\M$.
 Then for every $\epsilon>0$ and every $t$ with $0\leq t\leq 1$, there is a
projection $P\in \M$, $\tau(P)=t$, such that
$\|TP-PTP\|_2<\epsilon$.

\end{Corollary}

The following lemma extends Lemma~\ref{L:one half} to
arbitrary type ${\rm II}_1$ factors.

\begin{Lemma}\label{L:one half for arbitrary factor} Let $\M$ be a type ${\rm II}_1$ factor and let $T\in\M$.
 Then for every $\epsilon>0$, there is a
projection $P\in \M$, $\tau(P)=1/2$, such that
$\|TP-PTP\|_2<\epsilon$.
\end{Lemma}
\begin{proof}  Let $\N$ be the von
Neumann subalgebra generated by $T$. Then $\N$ is separable. If
$\N'\cap \M$ is a diffuse von Neumann algebra, then for every $t$,
$0\leq t\leq 1$, there is a projection $P\in \N'\cap \M$ such that
$PT=TP$ and $\tau(P)=t$. Hence Lemma~\ref{L:one half for arbitrary
factor} follows. If $\N'\cap \M$ is not a diffuse von Neumann
algebra, let $P_0,P_1,P_2,\cdots$ be a sequence of projections in
$\N'\cap \M$ such that $P_0+P_1+P_2+\cdots =I$, $P_0(\N'\cap
\M)P_0$ is diffuse, and $P_1,P_2,\cdots$ are non-zero minimal
projections in $(1-P_0)(\N'\cap \M)(1-P_0)$. Note that $(\N
P_n)'\cap (P_n\M P_n)=P_n(\N'\cap \M)P_n=\cc P_n$ for $n\geq 1$.
This implies that $\N P_n$ is a separable type ${\rm II}_1$ factor
for $n\geq 1$. There is an $n\geq 0$ such that
$\sum_{k=1}^n\tau(P_k)\leq t \leq \sum_{k=1}^{n+1}\tau(P_k)$.
Applying Corollary~\ref{C:arbitrary t} to $\N P_{n+1}$,
$t'=t-\sum_{k=1}^n\tau(P_k)$, and $T P_{n+1}$, there is a
projection $Q_{n+1}\in \N P_{n+1}$ such that $\tau(Q_{n+1})=t'$
and
\[\|TP_{n+1}Q_{n+1}-Q_{n+1}TP_{n+1}Q_{n+1}\|_2<\epsilon.
\] Let $P=P_0+P_1+\cdots+P_n+Q_{n+1}$. Then $P\in \M$,
$\tau(P)=t$, and
\[\|TP-PTP\|_2<\epsilon.
\]

\end{proof}

As a consequence of Lemma~\ref{L:one half for arbitrary factor},
Lemma~\ref{L:general r} is also true for arbitrary type ${\rm
II}_1$ factors.

\begin{proof}[Proof of Theorem~\ref{T:main result}]
 Let $T=(T_n)\in \M^\omega$.  By Lemma~\ref{L:general
 r}, for each $n$, there are projections $\{P_{n,j}\}_{0\leq j\leq 2^n}$ in $\M$
 such that $0=P_{n,0}<P_{n,1}< P_{n,2}< \cdots<P_{n,2^n-1}< P_{n,2^n}=I$,
 $\tau(P_{n,j})=j/2^n$, and $\|T_nP_{n,j}-P_{n,j}T_nP_{n,j}\|_2\leq 1/n$ for all $0\leq j\leq 2^n$.
For every $t$, $0\leq t\leq 1$, choose $P_{n,j}$ such that
$\tau(P_{n,j})\leq t<\tau(P_{n,j+1})$.  Let $P_t=(P_{n,j})\in
\M^\omega$. Then $P_s\leq P_t$ if $s\leq t$,
$\tau_\omega(P_t)=t$, and $TP_t=P_tTP_t$.\\

\end{proof}

\section{Operators with non-trivial invariant subspaces relative to a type ${\rm II}_1$ factor}

Let $\M$ be a type ${\rm II}_1$ factor with a faithful normal
tracial state $\tau$, and let $\mathfrak{M}=\{S'\in
\M:\,\text{there is a family of projections}\, \{P_t\}_{0\leq
t\leq 1}\, \text{in}\, \M\, \text{such that}\, ZP_t=P_tZP_t,
P_s\leq P_t\, \text{if}\, s\leq t,\, \text{and}\, \tau(P_t)=t\}$.
Let $(\M)_1$ be the set of operators $T$ in $\M$ such that
$\|T\|\leq 1$.  As an application of Theorem~\ref{T:main result},
we prove the following result.

\begin{Theorem}\label{T:dense} For every operator $T\in (\M)_1$ and every $\epsilon>0$,
there is an operator $S\in \mathfrak{M}\cap (\M)_1$ such that
$\|T-S\|_2<\epsilon$. In particular, the set $\mathfrak{M}$ is
dense in $\M$ in the strong operator topology.
\end{Theorem}

To prove Theorem~\ref{T:dense}, we need the following lemmas. The following lemma is well known.
\begin{Lemma}\label{L: unit ball is complete} Suppose
$\{T_n\}_n\subseteq (\M)_1$ is a Cauchy sequence with respect to
$\|\cdot\|_2$. Then there is an operator $T\in (\M)_1$ such that
\[\lim_{n\rightarrow\infty}\|T_n-T\|_2=0.
\]
\end{Lemma}

For an operator $T\in \M$, let $N(T)$ be the projection onto the
kernel space of $T$.
\begin{Lemma}\label{L:norm and 2norm} Let $\epsilon, \delta>0$ and $T\in \M$.
 If $\|T\|_2<\delta$, then there is a projection $P\in \M$ such
that $P\geq N(T)$, $\|TP\|\leq \epsilon$, and
$\tau(I-P)<\delta^2/\epsilon^2$.
\end{Lemma}
\begin{proof}By applying the polar decomposition theorem, we may assume
that $T$ is a positive operator. Let $\nu$ be the Borel measure on
$[0,\infty)$ induced by the composition of $\tau$ with the
spectral projections of $T$. Then
\[\|T\|_2^2=\int_0^\infty t^2d\nu(t)<\delta^2.
\] Let  $P=\chi_{[0,\epsilon]}(T)$. Then $P\geq N(T)$, $\|TP\|\leq
\epsilon$ and
\[\epsilon^2\tau(I-P)\leq \int_\epsilon^\infty t^2 d\nu(t)\leq
\|T\|_2^2< \delta^2.
\] Hence, $\tau(I-P)<\delta^2/\epsilon^2$.
\end{proof}

\begin{Lemma}\label{L:approximate by half} For every operator $T\in (\M)_1$ and every $\epsilon>0$,
there is an operator $S\in  (\M)_1$ such that
\begin{enumerate}
\item $\|T-S\|_2<\epsilon$ and

\item there is a projection $P\in \M$ such that $\tau(P)=1/2$ and
$SP=PSP$.
\end{enumerate}
\end{Lemma}
\begin{proof} Choose $\delta,\epsilon_1>0$ such that
\[\epsilon_1+
 \epsilon_1/\delta+\delta<\epsilon.\] By Corollary~\ref{C:main result}, there is a projection $P_1$
in $\M$ such that
\begin{equation}\label{E:delta}
\|TP_1-P_1TP_1\|_2<\delta.
\end{equation}
 Let $P_2=I-P_1$ and $T_{ij}=P_iTP_j$ for $i,j=1,2$. Then we can
 write
 \[T=\left(\begin{array}{cc}
 T_{11}&T_{12}\\
 T_{21}&T_{22}
 \end{array}\right)
 \] with respect to the decomposition $I=P_1+P_2$.
 Since $\|T\|\leq 1$, $\|T_{ij}\|\leq 1$ for all $i,j=1,2$. Note
 that
 (~\ref{E:delta}) implies  $\|T_{21}\|_2<\delta$ and also note that $N(T_{2,1})\geq
 P_2$. By Lemma~\ref{L:norm and 2norm}, there is a projection
 $Q\in \M$, $Q\geq P_2$, $\|T_{21}Q\|\leq \epsilon_1$ and
 $\tau(I-Q)< \epsilon_1^2/\delta^2$. Write $Q=P_1'+P_2$. Then
 $P_1'\leq P_1$ and $\tau(P_1-P_1')< \epsilon_1^2/\delta^2$. \\

 Let $R=T_{11}P_1'+T_{12}+T_{22}$, i.e.,  we can write
 \[R=\left(\begin{array}{cc}
 T_{11}P_1'&T_{12}\\
 0&T_{22}
 \end{array}\right)
 \] with respect to the decomposition $I=P_1+P_2$. Then
 $R-TQ=T_{21}Q$. Therefore,
 \begin{equation}\label{E:norm}
 \|R\|=\|TQ+T_{21}Q\|\leq 1+ \epsilon_1.
 \end{equation} On the other hand, $R-T=T_{11}(P_1-P_1')+T_{21}$. This implies
 that
 \begin{equation}\label{E:two norm}
 \|R-T\|_2\leq \|T_{11}(P_1-P_1')\|_2+\|T_{21}\|_2\leq
 \epsilon_1/\delta+\delta.
 \end{equation}

 Let $S=(1+\epsilon_1)^{-1} R$. Then (~\ref{E:norm})
 implies that $\|S\|\leq 1$ and (\ref{E:two norm}) implies that
 \[\|S-T\|_2\leq \|S-R\|_2+\|R-T\|_2\leq
 \epsilon_1\|S\|_2+  \epsilon_1/\delta+\delta\leq  \epsilon_1+
 \epsilon_1/\delta+\delta<\epsilon.
 \] Note that $SP_1=P_1SP_1$ and $\tau(P_1)=1/2$. Let $P=P_1$.
 We prove the lemma.

\end{proof}

\begin{proof}[Proof of Theorem~\ref{T:dense}] We use the induction
to construct operators $T_n$ and $\{P_{n,j}\}_{j=1}^{2^n}$ for
each $n\geq 0$ satisfying the following conditions:
\begin{enumerate}
\item for each $n$, $\{P_{n,j}\}_{j=1}^{2^n}$ is a family of
projections in $\M$ such that $\sum_{j=1}^{2^n}P_{n,j}=I$ and
$\tau(P_{n,j})=1/2^n$ for $1\leq j\leq 2^n$;

\item $P_{n,j}=P_{n+1,2j-1}+P_{n+1,2j}$ for $1\leq j\leq 2^n$;

\item $\|T_n\|\leq 1$, $T_0=T$, and
$\|T_n-T_{n+1}\|_2<\epsilon/2^{n+1}$;

\item  for each $k$, $1\leq k\leq 2^n$, $\sum_{j=1}^k P_{n,j}$ is
an invariant subspace of $T_n$.
\end{enumerate}
\vskip 1cm

For $n=0$, let $T_0=T$ and $P_{0,1}=I$. For $n=1$, by
Lemma~\ref{L:approximate by half}, there is an operator $S\in \M$,
$\|S\|\leq 1$, $\|S-T\|_2<\epsilon/2$ and there is a projection
$P\in \M$, $\tau(P)=1/2$ and $SP=PSP$. Let $T_1=S$, $P_{1,1}=P$
and $P_{1,2}=I-P$. Now for $n=2$, we construct $T_2$ and
$\{P_{2,j}\}_{j=1}^4$ satisfying the above conditions 1,2,3 and 4.
\vskip 1cm

Since $P_{1,1}$ is an invariant subspace of $T_1$, we can write
\[T_1=\left(\begin{array}{cc}
 A&T_{12}\\
 0&B
 \end{array}\right)
 \] with respect to the decomposition $I=P_{1,1}+P_{1,2}$. Let
 $\epsilon_1,\delta>0$ such that
 \[\epsilon_1+3\epsilon_1/\delta+2\delta<\epsilon/4.
 \] Applying Corollary~\ref{C:main result} to $A\in
 P_{1,1}\M P_{1,1}$ and $B\in P_{1,2}\M P_{1,2}$, there are
 projections $Q_1,Q_2,Q_3,Q_4$ such that $\tau(Q_j)=1/4$ for $1\leq j\leq 4$, $Q_1+Q_2=P_{1,1}$,
 $Q_3+Q_4=P_{1,2}$, $\|AQ_1-Q_1AQ_1\|_2<\delta$ and
 $\|BQ_3-Q_3BQ_3\|_2<\delta$. Now we can write
 \[T_1=\left(\begin{array}{cc}
\left(\begin{array}{cc} A_{11}&A_{12}\\
A_{21}&A_{22}
\end{array}\right) &T_{12}\\
0&\left(\begin{array}{cc} B_{11}&B_{12}\\
B_{21}&B_{22}
\end{array}\right)
\end{array}\right)
 \] with respect to the decomposition $I=Q_1+Q_2+Q_3+Q_4$.
Note that $\|AQ_1-Q_1AQ_1\|_2<\delta$ implies
$\|A_{21}\|_2<\delta$ and $\|BQ_3-Q_3BQ_3\|_2<\delta$ implies
$\|B_{21}\|_2<\delta$.  By Lemma~\ref{L:norm and 2norm} and
similar arguments as the proof of Lemma~\ref{L:approximate by
half}, there are projections  $Q_1'\leq Q_1$, $Q_3'\leq Q_3$ such
that $\|A_{21}Q_1'\|<\epsilon_1$, $\|B_{21}Q_3'\|<\epsilon_1$,
 $\tau(Q_1-Q_1')\leq \epsilon_1^2/\delta^2$ and $\tau(Q_3-Q_3')\leq
 \epsilon_1^2/\delta^2$.
\vskip 1cm

Let
\[R=\left(\begin{array}{cc}
\left(\begin{array}{cc} A_{11}Q_1'&A_{12}\\
0&A_{22}
\end{array}\right) &T_{12}(Q_3'+Q_4)\\
0&\left(\begin{array}{cc} B_{11}Q_3'&B_{12}\\
0&B_{22}
\end{array}\right)
\end{array}\right)
 \] with respect to the decomposition $I=Q_1+Q_2+Q_3+Q_4$. Then
\[\|R-T_1(Q_1'+Q_2+Q_3'+Q_4)\|=\|A_{21}Q_1'+B_{21}Q_3'\|<\epsilon_1
\]
and
\[\|R-T_1\|_2=\|A_{11}(Q_1-Q_1')+A_{21}+T_{12}(Q_3-Q_3')+B_{11}(Q_3-Q_3')+B_{21}\|_2\leq
3\epsilon_1/\delta+2\delta.
\]
Therefore, \[\|R\|\leq
\|T_1(Q_1'+Q_2+Q_3'+Q_4)\|+\|R-T_1(Q_1'+Q_2+Q_3'+Q_4)\|<
1+\epsilon_1.\] Let $T_2=(1+\epsilon_1)^{-1}R$.
 Then $\|T_2\|\leq 1$ and \[
 \|T_2-T_1\|_2\leq
 \|T_2-R\|_2+\|R-T_1\|_2<\epsilon_1\|T_2\|+3\epsilon_1/\delta+2\delta<
 \epsilon_1+3\epsilon_1/\delta+2\delta<\epsilon/4.\] Let
 $P_{2,j}=Q_j$ for $1\leq j\leq 4$. Then $T_2$ and $\{P_{2,j}\}_{j=1}^4$ satisfy the
 conditions 1,2,3 and 4. The general case can be proved similarly
 by using the induction.

 \vskip 1cm
 Suppose  $T_n$ and $\{P_{n,j}\}_{j=1}^{2^n}$ satisfy the above
conditions 1,2,3 and 4. By 3 and Lemma~\ref{L: unit ball is
complete}, there is an operator $S\in (\M)_1$ such that
$\lim_{n\rightarrow\infty}\|S-T_n\|_2=0$ and $\|S-T\|_2<
\epsilon$. By 2 and 4, for each $n$ and $k$, $1\leq k\leq 2^n$,
$\sum_{j=1}^k P_{n,j}$ is an invariant subspace of $T_N$ for
$N\geq n$ and therefore an invariant subspace of $S$. By 1,
$\tau(\sum_{j=1}^k P_{n,j})=k/2^n$. Note that $\{k/2^n:n\geq 0,
1\leq k\leq 2^n\}$ is dense in $[0,1]$. For every $t$, $0\leq
t\leq 1$, let \[\displaystyle{P_t=\bigvee_{k/2^n\leq
t}\left(\sum_{j=1}^k P_{n,j}\right)}.\] By 1, $P_s\leq P_t$ if
$s\leq t$, $\tau(P_t)=t$ and $SP_t=P_tSP_t$.

\end{proof}

\section{$\prod^\omega M_n(\cc)$ is not $\ast$-isomorphic to $\R^\omega$}

Throughout this section $\M$ is a separable type ${\rm II}_1$
factor. Recall that a separable type ${\rm II}_1$ factor $\M$ has
property $\Gamma$ if for every $n$, $T_1,\cdots, T_n\in \M$, and
every $\epsilon>0$, there is a projection $P\in \M$ such that
$\tau(P)=1/2$ and $\|T_iP-PT_i\|_2<\epsilon$ (cf.~\cite{Di}).

\begin{Lemma}\label{L:operators in ultrapower of gamma}
 Suppose $\M$ has property $\Gamma$. Then for every
operator $T\in \M^\omega$ and $t$, $0\leq t\leq 1$, there is a
projection $P\in \M^\omega$ such that $PT=TP$ and
$\tau_\omega(P)=1/2$.
\end{Lemma}
\begin{proof} Write $T=(T_n)$. Since $\M$ has property $\Gamma$,
there exists a projection $P_n\in \M$ such that
$\|P_nT_n-T_nP_n\|_2<1/n$ and $\tau(P_n)=1/2$. Let $P=(P_n)\in
\M^\omega$. Then $PT=TP$ and $\tau_\omega(P_n)=1/2$.
\end{proof}

  Let  $(M_n(\cc))_1$ be the set of
matrices $T\in M_n(\cc)$ such that $\|T\|\leq 1$,  and let
$\nu((M_n(\cc))_1,\omega)$ be the covering number of
$(M_n(\cc))_1$ with respect to the normalized trace norm
$\|\cdot\|_2$. There are universal constants
$c_1,c_2$~\cite{Sz1,Sz2} such that
\begin{equation}\label{E:covering number of unit ball}
\left(\frac{c_1}{\omega}\right)^{2n^2}\leq
\nu((M_n(\cc))_1,\omega)\leq
\left(\frac{c_2}{\omega}\right)^{2n^2}.
\end{equation}
\vskip 1cm

The next lemma follows from  Theorem 9 of Herrero and
 Szarek~\cite{H-S} (also see~\cite{Von}). For the sake of completeness, we include a
direct proof.
\begin{Lemma}\label{L:H-S}There exists a universal constant
$\alpha>0$ with the following property: for each $n\geq 2$, there
exists a matrix $T_n\in M_n(\cc)$, $\|T_n\|=1$, such that
\[\|PT_n-T_nP\|_2\geq \alpha
\] for every  projection $P\in M_n(\cc)$ with
 ${\rm rank}P=\left[\frac{n}{2}\right]$, where
 $\left[\frac{n}{2}\right]$ is the maximal integer less or equal
 to $\frac{n}{2}$.
\end{Lemma}
\begin{proof} Suppose the lemma is false. Then for every
$\epsilon> 0$, there is an $n\geq 2$, for every matrix $T\in
M_n(\cc)$, $\|T\|\leq 1$, there is a projection $P\in M_n(\cc)$
such that ${\rm rank}P=\left[\frac{n}{2}\right]$ and
$\|PT-TP\|_2<\epsilon$. Without of loss of generality we may
assume that $n=2k$. Let $(M_n(\cc))_1$ be the set of $n\times n$
complex matrices $T$ such that $\|T\|\leq 1$. For $T\in M_n(\cc)$,
let $\|T\|_2$ be the trace norm with respect to the normalized
trace $\tau_n=\frac{Tr}{n}$ on $M_n(\cc)$. \vskip 1cm

 By (\ref{E:covering number of unit ball}),
\begin{equation}\label{E:lower bound}
\left(\frac{c_1}{2\epsilon}\right)^{2n^2}\leq
\nu((M_n(\cc))_1,2\epsilon)\leq
\left(\frac{c_2}{2\epsilon}\right)^{2n^2} \end{equation} and
\[\left(\frac{c_1}{\epsilon}\right)^{2k^2}\leq
\nu((M_k(\cc))_1,\epsilon)\leq
\left(\frac{c_2}{\epsilon}\right)^{2k^2}.
\]
Let $\{T_{t}\}_{t\in \mathbb{T}}$ be an $\epsilon$-net of
$(M_k(\cc))_1$ such that $\#\mathbb{T}\leq
\left(\frac{c_2}{\epsilon}\right)$. \vskip 1cm

 Now for every $T\in (M_n(\cc))_1$,
$\|TP-PT\|_2<\epsilon$ for some projection $P\in M_n(\cc)$ with
rank $k$. Write
\[T=\left(\begin{array}{cc}
T_{11}&T_{12}\\
T_{21}&T_{22}
\end{array}\right)
\] with respect to the decomposition $I=P+(I-P)$.  Since $\|T\|\leq 1$, $\|T_{11}\|,
\|T_{22}\|\leq 1$. Choose $t_1,t_2\in \mathbb{T}$ such that
$\|T_{11}-T_{t_1}\|_2<\epsilon$ and
$\|T_{22}-T_{t_2}\|_2<\epsilon$ with respect to the normalized
trace norm on $M_k(\cc)$. Since $\|TP-PT\|_2<\epsilon$,
\[\|T-\left(\begin{array}{cc}
T_{t_1}&0\\
0&T_{t_2}
\end{array}\right)\|_2<2\epsilon.
\]
This implies that,
\begin{equation}\label{E:covering number of Omega j}
\nu((M_n(\cc))_1, 2\epsilon)\leq
\left(\frac{c_2}{\epsilon}\right)^{2k^2}\cdot
\left(\frac{c_2}{\epsilon}\right)^{2k^2}=\left(\frac{c_2}{\epsilon}\right)^{4k^2}.
\end{equation}
\vskip 1cm

 Note that $n=2k$. By (\ref{E:lower bound}),
\[
\left(\frac{c_1}{2\epsilon}\right)^{2n^2}\leq
\left(\frac{c_2}{\epsilon}\right)^{n^2}.
\]
By taking $\ln$ on both sides, we have
\[\frac{2(\ln c_1-\ln 2-\ln \epsilon )}{-\ln \epsilon}\leq
\frac{\ln c_2-\ln \epsilon}{-\ln \epsilon}.\] Let
$\epsilon\rightarrow 0+$. This implies  $2\leq 1$. This is a
contradiction.
 \end{proof}

\begin{Theorem}\label{T:ultrapower} The von Neumann algebra
$\prod^\omega M_n(\cc)$ is not $\ast$-isomorphic to $\R^\omega$,
the ultrapower algebra of the hyperfinite ${\rm II}_1$ factor.
\end{Theorem}
\begin{proof} Choose $T_n\in M_n(\cc)$ as in Lemma~\ref{L:H-S}.
Let $T=(T_n)\in \prod^\omega M_n(\cc)$. Claim if $P$ is a
projection in $\prod^\omega M_n(\cc)$ such that $TP=PT$, then
$\tau_\omega(P)\neq 1/2$. Otherwise, suppose $P=(P_n)\in
\prod^\omega M_n(\cc)$ is a projection such that $TP=PT$ and
$\tau_\omega(P)=1/2$. We may assume that $P_n$ is a projection in
$M_n(\cc)$ with ${\rm rank}P=\left[\frac{n}{2}\right]$. By
Lemma~\ref{L:H-S}, $\|T_nP_n-P_nT_n\|_2\geq \alpha>0$. Hence
$\|PT-TP\|_2\geq \alpha>0$. This is a contradiction.
 On the other hand, for every operator $T\in \R^\omega$,
there is a projection $Q\in \R^\omega$ such that $TQ=QT$ and
$\tau_\omega(Q)=1/2$ by Lemma~\ref{L:operators in ultrapower of
gamma}. So $\prod^\omega M_n(\cc)$ is not $\ast$-isomorphic to
$\R^\omega$.
\end{proof}

\begin{Remark}\emph{By Theorem 9 of~\cite{H-S}, there is an operator $T$
in $\prod^\omega M_n(\cc)$ such that if $TP=PT$ for some
projection $P$ in $\prod^\omega M_n(\cc)$, then $P=0$ or $P=I$.}
\end{Remark}

\noindent{\bf Question:}\, Can $\R^\omega$ be embedded into
$\prod^\omega M_n(\cc)$? If $\M$ is a separable type ${\rm II}_1$
factor and $\M^\omega\cong \R^\omega$, is $\M\cong \R$?

\section{The lattice of invariant subspaces of an operator affiliated with a type ${\rm II}_1$ factor }

Let $\M$ be a factor (not necessarily type ${\rm II}_1$) acting on
a Hilbert space $\H$ and $T\in\M$. We denote by $Lat_{\M} T$ the
set of projections $P\in \M$ such that $TP=PTP$. So $P\in
Lat_{\M}$ if and only if $P\H$ is an invariant subspace of $T$.
 Recall that a hyperinvariant subspace of $T$ is a
(closed) subspace invariant under every operator in $\{T\}'$. It
is easy to see that the projection onto a hyperinvariant subspace
of $T$ is in the von Neumann algebra generated by $T$.

Suppose $S,T$ are two operators in $\M$. Recall that $S$ and $T$
are \emph{quasi-similar} in $\M$ if there are operators $X, Y\in
\M$ which are one-to-one and have dense range such that $SX=XT$
and $YS=TY$. The following theorem is given in~\cite{R-R}(Theorem
6.19).

\begin{Theorem} If $S$
and $T$ are quasi-similar in $\B(\H)$ and $S$ has a nontrivial
hyperinvariant subspace, then $T$ has a nontrivial hyperinvariant
subspace.
\end{Theorem}

It is still not known that if  we replace the hyperinvariant
subspace by the invariant subspace in the above theorem, the
theorem still holds or not. However, in this section we will show
that if we replace $\B(\H)$ by a type ${\rm II}_1$ factor  and
replace the hyperinvariant subspace by the invariant subspace,
then the above theorem still holds.\vskip 1cm

 We denote by $N(T)$ the
kernel space of $T$ and $R(T)$ the closure of range space of $T$.

\begin{Lemma}\label{L:kernal and range}
 Let $\M$ be a finite von Neumann algebra with a faithful normal trace
$\tau$, and let $T\in \M$. Then $\tau(R(T))+\tau(N(T))=1$. In
particular, $N(T)=0$ if and only if $R(T)=I$.
\end{Lemma}
\begin{proof} By the polar decomposition theorem,
there is a unitary operator $U$ and a positive operator $|T|$ in
$\M$ such that $T=U|T|$. So $T^*=|T|U^*$. Now, we have
$T^*T=|T|^2=U^*TT^*U$. Thus,
$\tau(R(T))=\tau(R(TT^*))=\tau(R(T^*T))=\tau(R(T^*))=1-\tau(N(T))$.
\end{proof}

\begin{Corollary}\label{C:kernal and range}  Let $\M$ be a finite von Neumann algebra
 with a faithful normal trace
$\tau$. Let $T\in \M$ be an operator such that $N(T)=0$, and let
$E\in\M$ be a projection. Then $\tau(R(TE))=\tau(E)$. In
particular, if $0<E<I$, then $0<R(TE)<I$.
\end{Corollary}
\begin{proof} Since $N(T)=0$, $N(TE)=I-E$. By lemma~\ref{L:kernal and range},
$\tau(R(TE))=1-\tau(N(TE))=1-\tau(I-E)=\tau(E)$.
\end{proof}

\begin{Proposition}\label{P:quasi-similarity and invariant subspaces}
  Let $\M$ be a type ${\rm II}_1$ factor with a faithful normal trace
$\tau$ and $S, T\in\M$. If there is an operator $X\in\M$ such that
$N(X)=0$ and $XS=TX$, then $Lat S$ is isomorphic to a sublattice
of $Lat T$ and $Lat T$ is isomorphic to a sublattice of $Lat S$.
In particular,  $S$ has a nontrivial invariant subspace if and
only if $T$ has a nontrivial invariant subspace.
\end{Proposition}
\begin{proof}For $E\in Lat_\M S$, let $F=R(XE)$. The assumption $XS=TX$ implies that $F\in Lat_\M T$.
  Define $\phi(E)=F$.
 By corollary~\ref{C:kernal
and range}, $\tau(F)=\tau(E)$. We want to show that $\phi$ is a
lattice isomorphism from $Lat_\M S$ onto a sublattice of $Lat_\M
T$. Let $E_1, E_2\in Lat S$. Then $\phi(E_1\vee E_2)=R(X(E_1\vee
E_2))= R(XE_1)\vee R(XE_2)=\phi(E_1)\vee \phi(E_2)$ and
$\phi(E_1\wedge E_2)=R(X(E_1\wedge E_2))\leq R(X(E_1))\wedge
R(X(E_2))=\phi(E_1)\wedge \phi(E_2)$. By corollary~\ref{C:kernal
and range}, \[\tau(\phi(E_1)\wedge
\phi(E_2))=\tau(\phi(E_1)\vee\phi(E_2))-\tau(\phi(E_1))-\tau(\phi(E_2))\]\[=
\tau(E_1\vee E_2)-\tau(E_1)-\tau(E_2)=\tau(E_1\wedge
E_2)=\tau(\phi(E_1\wedge E_2)).\] So $\phi(E_1\wedge
E_2)=\phi(E_1)\wedge \phi(E_2)$. Thus $\phi$ is a lattice
homomorphism. Let $E_1, E_2\in Lat S$ and $E_1\neq E_2$. We may
assume that $E=E_1\vee E_2>E_1$. So $\tau(E)>\tau(E_1)$. If
$\phi(E_1)=\phi(E_2)=F\in Lat T$. Then $F=\phi(E_1\vee E_2)$. By
corollary~\ref{C:kernal and range},
$\tau(F)=\tau(E_1)=\tau(E_1\vee E_2)=\tau(E)$. This is a
contradiction. So $\phi$ is a lattice isomorphism from $Lat_\M S$
onto a sublattice of
$Lat_\M T$.\\

Similarly, by $X^*T^*=S^*X^*$, there is a lattice isomorphism from
$Lat_\M T^*$ onto a sublattice of $Lat_\M S^*$. Since $Lat_\M T$
is isomorphism to $Lat_\M T^*$ and $Lat_\M S$ is isomorphic  to
$Lat_\M S^*$. So there is a lattice isomorphic from $Lat_\M T$
onto a sublattice of $Lat_\M S$.
\end{proof}

\begin{Proposition}\label{P:quasi-similarity and hyperinvariant subspaces}
 Let $\M$ be a type ${\rm II}_1$ factor and $S,T\in \M$. If $S$
and $T$ are quasi-similar, then the lattice of hyperinvariant
subspaces of $S$ and the lattice of hyperinvariant subspaces of
$T$ are isomorphic.
\end{Proposition}
\begin{proof} Let $X, Y$ in $\M$ be one to one operators with dense
ranges such that $XS=TX$ and $SY=YT$. Let $E$ be a hyperinvariant
subspace of $S$. Let $F$ the closure of the linear span of
$R(AXE)$, where $AT=TA$. Then clearly $F$ is a hyperinvariant
subspace of $T$. Note that $\tau(F)\geq \tau(XE)=\tau(E)$ by
corollary~\ref{C:kernal and range}. Since $YAXS=YATX=YTAX=SYAX$
and $E$ is a hyperinvariant subspace of $S$, $R(YAXE)\leq E$ and
therefore, $R(YF)\leq E$. By corollary~\ref{C:kernal and range},
$\tau(E)\geq \tau(F)$. So $\tau(F)=\tau(E)$,  $F=R(XE)$, and
$E=R(YF)$. Now $E\rightarrow F=R(XE)$ is a lattice isomorphism
(the inverse is $F\rightarrow E=R(YF)$) from the lattice of
hyperinvariant subspaces of $S$ onto the lattice of hyperinvariant
subspaces of $T$.
\end{proof}

\begin{Corollary}\label{C:invariant and hyperinvariant}
 Let $\M$ be a type ${\rm II}_1$ factor and $S,T\in \M$. Then
 $Lat_\M ST$ is not trivial iff  $Lat_\M TS$ is not trivial. Furthermore, if $N(S)=N(T)=0$, then  $Lat_\M ST$
 is isomorphic to $Lat_\M TS$ and the lattice of hyperinvariant subspaces of $ST$ is isomorphic
 to the lattice of hyperinvariant subspaces of $TS$ as lattices.
\end{Corollary}
\begin{proof}  Suppose  $Lat_\M ST$ is not trivial. If $TS=0$, then
$Lat_\M  TS$ is not trivial. We assume that $TS\neq 0$.
 If $N(S)\neq 0$ or
$R(T)\neq I$, then $N(S)$ or $R(T)$ is a non trivial invariant
subspace of $TS$.  if $N(S)=0$ and $R(T)=I$, then by
lemma~\ref{L:kernal and range}, $R(S)=I$ and $N(T)=0$.  Thus
$ST,TS$ are quasisimilar. By Proposition~\ref{P:quasi-similarity
and invariant subspaces}, $Lat_\M TS$ is not trivial.\vskip 1 cm

If $N(S)=N(T)=0$, then $R(S)=R(T)=I$ by lemma~\ref{L:kernal and
range}. For $E\in Lat_\M ST$, let $F=R(TE)$ and $E_1=R(SF)$. Then
$E_1=R(SF)=R(STE)\leq E$ since $E\in Lat_\M ST$. By
corollary~\ref{C:kernal and range}, $\tau(E)=\tau(F)=\tau(E_1)$.
This implies that $E=E_1$. Note that $R(TSF)=R(TSTE)\leq R(TE)=F$,
$F\in Lat_\M TS$. Define $\phi(E)=R(TE)$ and $\psi(F)=R(SF)$ for
$E\in Lat_\M ST$ and $F\in Lat_\M TS$, respectively. Then
$\psi=\phi^{-1}$. So $\phi$ is a lattice isomorphism from $Lat_\M
ST$ onto $Lat_\M TS$.\vskip 1cm

The lattice of hyperinvariant subspaces of $ST$ is isomorphic
 to the lattice of hyperinvariant subspaces of $TS$ as lattices is a
 corollary of Proposition~\ref{P:quasi-similarity and hyperinvariant subspaces}.
\end{proof}

\begin{Remark}\label{R:last remark}\emph{
Let $T\in \B(H)$ and $V\in \B(H)$ such that $VV^*=I$ but $V^*V\neq
I$. Then  $R(V^*)$ is a nontrivial invariant subspace of $V^*TV$.
Note that $T=TVV^*$. If the first part of
Corollary~\ref{C:invariant and hyperinvariant} is true for
$\M=\B(\H)$, then the answer to the invariant subspace question
(relative to $\B(\H)$ is affirmative.}
\end{Remark}

 \vspace{.2in}
\noindent {\em E-mail address: } [Junsheng Fang]
jfang\@@math.tamu.edu\\
\noindent{\em Address:} Department of Mathematics, Texas A\&M University, College Station, TX,77843.\\

\noindent {\em E-mail address: } [Don Hadwin] don\@@math.unh.edu\\
\noindent{\em Address:} Department of Mathematics,  University of New Hampshire, Durham, NH, 03824.


\begin{thebibliography}{99}


\bibitem{Br} L.G. Brown,   Lidskii's theorem in the type ${\rm II}$ case,
Geometric methods in operator algebras, H. Araki and E. Effros
(Eds.) \emph{Pitman Res. notes in Math. Ser} {\bf 123}, Longman Sci.
Tech. (1986), 1-35.

\bibitem{D-H1} K. Dykema and U. Haagerup, Invariant subspaces of
Voiculescu's  circular operator, \emph{Geom. Funct. Anal.} {\bf
11} (2001), 693-741.
\bibitem{D-H2} K. Dykema and U. Haagerup, Invariant subspaces of the
quasinilpotent DT-operator, \emph{J. Funct. Anal.} {\bf 209} no.2,
(2004), 332-366.


\bibitem{Di} J. Dixmier,  Quelques propri\'et\'es des suites
centrales dans les facteurs de type ${\rm II}_{1}$, (French)
\emph{Invent. Math.}\ {\bf 7} (1969) 215--225.


\bibitem{H-S} D. A. Herrero and S. J. Szarek  How well can an $n\times n$ matrix be approximated by reducible ones?
\emph{Duke Math. J.} {\bf 53} (1986), no. 1, 233--248.

\bibitem{H-S1} U. Haagerup and H. Schultz, Invariant Subspaces for
Operators in a General ${\rm II}_1$-factor, preprint available at
http://www.arxiv.org/pdf/math.OA/0611256.

\bibitem{H-S2} U. Haagerup and H. Schultz, Brown measures of unbounded
operators affiliated with a finite von Neumann algebra, 
\emph{Math Scand.},  {\bf 100}  (2007),  no. 2, 209--263.

\bibitem{Mc} D. McDuff,  Central sequences and the hyperfinite
factor,  \emph{Proc.  London Math. Soc.} {\bf 21} (1970),
443--461.

\bibitem{M-v} F. Murray and J. von Neumann,  On rings of operators,
IV,\emph{ Ann. of Math.} {\bf 44} (1943), 716--808.

\bibitem{Po} S. Popa, Free independent sequences in type II1 fractors and related problems,
\emph{Asterisque}, {\bf 232} (1995), 187-202.

\bibitem{R-R} H. Radjavi and P. Rosenthal,  ``Invariant Subspaces'',
Springer-Verlag, New York, 1973.

\bibitem{Sa}  S. Sakai,  ``The Theory of W* Algebras'',
Lecture notes, Yale University, 1962.

\bibitem{S-S} P. Sniady and R. Speicher, Continuous family of invariant
subspaces for $R-$diagonal operators, \emph{Invent. Math.} {\bf
146} (2001), 329-363.
\bibitem{Sz1}  S. J. Szarek, Nets of Grassmann manifold and orthogonal group, \emph{Proceedings of research
workshop on Banach space theory} (Iowa City, Iowa, 1981),
169�C185, Univ. of Iowa, Iowa City, Iowa, 1982.


\bibitem{Sz2} S. J. Szarek, The finite-dimensional basis problem
 with an appendix on nets of Grassmann manifolds,  \emph{Acta Math.}  {\bf 151}  (1983),  no. 3-4, 153--179.

\bibitem{VDN} D.V. Voiculescu, K. Dykema and A. Nica, ``Free Random
Variables", CRM Monograph Series, vol. 1, AMS, Providence, R.I.,
1992.

\bibitem{Von} J. von Neumann, Approximative properties of matrices of high finite order,  
\emph{Portugaliae Math.}  {\bf 3},  (1942). 1--62.

\end{thebibliography}
\end{document}